\documentclass[11pt]{amsart}
\usepackage{amsfonts,amsthm,amsmath,amssymb,amscd,mathrsfs}
\usepackage{graphics}
\usepackage{indentfirst}
\usepackage{cite}
\usepackage{latexsym}
\usepackage[dvips]{epsfig}
\usepackage[utf8]{inputenc}
\usepackage{color}
\usepackage{esint}
\usepackage{tikz}
\usepackage[colorlinks=true,linkcolor=blue]{hyperref}
\usepackage{verbatim}
\usepackage{bm}

\numberwithin{equation}{section}

\newtheorem{theorem}{Theorem}[section]
\newtheorem{lemma}[theorem]{Lemma}
\newtheorem{corollary}[theorem]{Corollary}
\newtheorem{proposition}[theorem]{Proposition}

\theoremstyle{definition}
\newtheorem{definition}{Definition}
\newtheorem{assumption}{Assumption}
\newtheorem{example}{Example}

\theoremstyle{remark}
\newtheorem{remark}{Remark}

\def\be{\begin{eqnarray}}
\def\ee{\end{eqnarray}}
\def\ba{\begin{align}}
\def\ea{\end{align}}
\def\bay{\begin{array}}
\def\eay{\end{array}}
\def\bca{\begin{cases}}
\def\eca{\end{cases}}
\def\p{\partial}

\def\no{\nonumber}

\def\la{\lambda}

\def\ka{\kappa}
\def\al{\alpha}

\def\bt{\begin{theorem}}
\def\et{\end{theorem}}
\def\bc{\begin{corollary}}
\def\ec{\end{corollary}}
\def\bl{\begin{lemma}}
\def\el{\end{lemma}}
\def\bp{\begin{proposition}}
\def\ep{\end{proposition}}
\def\br{\begin{remark}}
\def\er{\end{remark}}
\def\bd{\begin{definition}}
\def\ed{\end{definition}}
\def\bpf{\begin{proof}}
\def\epf{\end{proof}}
\def\bex{\begin{example}}
\def\eex{\end{example}}
\def\bq{\begin{question}}
\def\eq{\end{question}}
\def\bas{\begin{assumption}}
\def\eas{\end{assumption}}
\def\ber{\begin{exercise}}
\def\eer{\end{exercise}}

\def\u{{\textbf u}}
\def\v{{\textbf v}}

\def\x{{\textbf x}}

\def\V{{\mathcal V}}

\def\B{{\mathcal B}}

\begin{document}
\title[Transonic Shock Flows]{Structural Stability of Transonic Shock Flows with an External Force}
\author{Shangkun Weng}
\address{School of Mathematics and Statistics, Wuhan University, Wuhan, Hubei Province, 430072, People's Republic of China.}
\email{skweng@whu.edu.cn}
\author{Wengang Yang}
\address{School of Mathematics and Statistics, Wuhan University, Wuhan, Hubei
	Province, 430072, People's Republic of China.}
\email{yangwg@whu.edu.cn}

\begin{abstract}
  This paper is devoted to the structural stability of a transonic shock passing 
  through a flat nozzle for two-dimensional steady compressible flows with an 
  external force. We first establish the existence and uniqueness of one dimensional 
  transonic shock solutions to the steady Euler system with an external force by 
  prescribing suitable pressure at the exit of the nozzle when the upstream flow is 
  a uniform supersonic flow. It is shown that the external force helps to stabilize 
  the transonic shock in flat nozzles and the shock position is uniquely determined. 
  Then we are concerned with the structural stability of these transonic shock 
  solutions when the exit pressure is suitably perturbed. One of the new ingredients 
  in our analysis is to use the deformation-curl decomposition to the steady Euler 
  system developed in \cite{WengX2019} to deal with the transonic shock problem.
\end{abstract}

\date{}
\maketitle

\section{Introduction and main results} \noindent
\par The studies of transonic shock solutions for inviscid compressible flows in different kinds of nozzles had a long history and had obtained many important new achievements during the past twenty years. Courant and Friedrichs \cite{CF} had described the transonic shock
phenomena in a de Laval nozzle whose cross section decreases first and then increases. It was observed in experiment that if the
upcoming flow becomes supersonic after passing through the throat of the nozzle, to match the prescribed appropriately large exit pressure, a shock front intervenes at some place in the diverging part of the nozzle and the gas is compressed and slowed down to subsonic speed.

People first used the quasi one dimensional model to study the transonic shock problem \cite{L1958,CF,egm84,L1982}. The structural stability of multidimensional transonic shocks in flat or diverging nozzles were further investigated in \cite{ChenF, XY2005, XinY2008} using the steady potential flows with different kinds of boundary conditions. In particular, \cite{XY2005, XinY2008} proved that the stability of transonic shocks for potential flows is usually ill-posed under the perturbations of the exit pressure. Many researchers also considered the transonic shock problem in the flat or almost flat nozzles with the exit pressure satisfying some special constraint, see \cite{C2005,C2008,cy08,LY2008,XYY2009} and the references therein. Recently, there is an interesting progress on the stability and existence of transonic shock solutions to the two dimensional and three dimensional axisymmetric steady compressible Euler system in an almost flat finite nozzle with the receiver pressure prescribed at the exit of the nozzle (see \cite{FangXin, FangGao}), where the shock position was uniquely determined.

The structural stability of the transonic shock problem in two dimensional divergent
nozzles under the perturbations for the exit pressure was first established in
\cite{LXY2009a} when the opening angle of the nozzle is suitably small. Later on,
this restriction was removed in \cite{LXY2009b,LXY2013}. Furthermore, the transonic
shock in general two dimensional straight divergent nozzles was shown in
\cite{LXY2013} to be structurally stable under generic perturbations for both the
nozzle shape and the exit pressure. The existence and stability of transonic shock
for three dimensional axisymmetric flows without swirl in a conic straight nozzle
were established in \cite{lxy10a,lxy10b} with respect to small perturbations of the
exit pressure. For the structural stability under the axisymmetric perturbation of
the nozzle wall, a modified Lagrangian coordinate was introduced in \cite{WengXX} to
deal with the corner singularities near the intersection points of the shock surface
and nozzle boundary and the artificial singularity near the axis simultaneously.
Most recently, the authors in \cite{WengXY2021a,WengXY2021b} studied radially
symmetric transonic flow with/without shock in an annulus. Thanks to the effect of
angular velocity, it was found in \cite{WengXY2021a} that besides the well-known
supersonic-subsonic shock in a divergent nozzle as in the case without angular
velocity, there exists a supersonic-supersonic shock solution, where the downstream
state may change smoothly from supersonic to subsonic. Furthermore, there exists a
supersonic-sonic shock solution where the shock circle and the sonic circle coincide.


In this paper, we will consider similar transonic shock
phenomena occurring in a flat nozzle when the fluid is exerted with an external force. The 2-D steady compressible isentropic Euler system with external force are of the form
\begin{eqnarray}\label{2deuler-force}
\begin{cases}
\p_{x_1} (\rho u_1)+ \p_{x_2}(\rho u_2)=0,\\
\p_{x_1} (\rho u_1^2+P(\rho))+ \p_{x_2}(\rho u_1 u_2)=\rho \partial_{x_1} \Phi,\\
\p_{x_1} (\rho u_1u_2)+ \p_{x_2}(\rho u_2^2+P(\rho))=\rho \p_{x_2} \Phi,
\end{cases}
\end{eqnarray}
where $(u_1,u_2)={\bf u}:\mathbb{R}^2\rightarrow\mathbb{R}^2$ is the unknown
velocity filed  and  $\rho:\mathbb{R}^2\rightarrow\mathbb{R}$ is the density, and
$\Phi(x_1,x_2)$ is a given potential function of external force. For the ideal
polytropic gas, the equation of state is given by $P(\rho)=A \rho^{\gamma}$, here
$A$ and $\gamma$ ($1<\gamma<3$) are positive constants. We take $A=1$ throughout
this paper for the convenience.

To this end, let's firstly focus on the 1-D steady compressible flow with an external force on an interval $I=[L_0,L_1]$,
which is governed by
\begin{eqnarray}\label{1df}\begin{cases}
		(\bar{\rho} \bar{u})'(x_1)=0,\\
		\bar{\rho} \bar{u} \bar{u}'+ \frac{d}{dx_1} P(\bar{\rho})= \bar{\rho}
		\bar{f}(x_1),\\
		\bar{\rho}(L_0)=\rho_0>0,\ \ \bar{u}(L_0)= u_0>0,
\end{cases}\end{eqnarray}
where we assume that the flow state at the entrance $x_1=L_0$ is supersonic, meaning that $u_0^2>c^2(\rho_0)=\gamma
\rho_0^{\gamma-1}$. 

Denote $J= \bar{\rho} \bar{u}=\rho_0 u_0>0$, then it follows from \eqref{1df}
that
\be\label{1df1}\begin{cases}
	\bar{\rho}(x_1)=\frac{J}{\bar{u}(x_1)},\\
	((\bar{u})^{\gamma+1}-\gamma J^{\gamma-1}) \bar{u}'= \bar{u}^{\gamma} \bar{f}.
\end{cases}\ee
Also one has
\be\label{1df001}
&&\bar{u}'= \frac{\bar{u}\bar{f}}{\bar{u}^2-c^2(\bar{\rho})},\quad
\bar{\rho}'=-\frac{\bar{\rho}\bar{f}}{\bar{u}^2-c^2(\bar{\rho})},\\\label{1df002}
&&\frac{d}{dx_1} \bar{M}^2(x_1)=\frac{(\gamma+1) \bar{M}^2}{\bar{M}^2-1} \frac{\bar{f}}{c^2(\bar{\rho})},
\ee
where $\bar{M}(x_1)=\frac{\bar{u}(x_1)}{c(\bar{\rho})}$ is the Mach number.

Since $\bar{M}^2(L_0)>1$, it follows from
\eqref{1df002} that if the external force satisfies
\be\label{1df04}
\bar{f}(x_1)>0, \ \ \forall L_0<x_1<L_1,
\ee
then the problem \eqref{1df} has a global supersonic solution $(\bar{\rho}^-, \bar{u}^-)$ on $[L_0,L_1]$. If one prescribes a large enough end pressure at $x_1=L_1$, a shock will form at some point $x_1=L_s\in (L_0,L_1)$ and the gas is compressed and slowed down to
subsonic speed, the gas pressure will increase to match the given end pressure.
Mathematically, one looks for a shock $x_1=L_s$ and smooth functions $(\bar{\rho}^{\pm}, \bar{u}^{\pm},
\bar{P}^{\pm})$ defined on $I^+=[L_{s}, L_1]$ and $I^-=[L_0,L_s]$ respectively, which solves \eqref{1df1} on $I^{\pm}$ with the jump
at the shock $x_1=L_{s}\in (L_0, L_1)$ satisfying the physical entropy condition
$[\bar{P}(L_{s})]=\bar{P}^+(L_s)-\bar{P}^-(L_s)>0$ and the Rankine-Hugoniot conditions
\be\label{1df11}\label{rh}\begin{cases}
	[{\bar \rho} {\bar u}](L_s)=0,\\
	[{\bar\rho} {\bar u}^2+P({\bar\rho})](L_s)=0.
\end{cases}\ee
and also the boundary conditions
\begin{eqnarray}\label{bd1}
	&&\rho(L_0)=\rho_{0},\ u(L_0)=u_{0}>0,\\\label{bd2}
	&&\bar{P}(L_1)= P_{e}.
\end{eqnarray}

We will show that there is a unique transonic shock solution to the 1-D
Euler system when the end pressure $P_e$ lies in a suitable interval. Such a
problem will be solved by a shooting method employing the monotonicity relation between the shock position and the end pressure.

\begin{lemma}\label{1dtransonic-shock} Suppose that the initial state $(u_0,\rho_0)$
at $x_1=L_0$ is supersonic and the external force $f$ satisfying \eqref{1df04},
there exists two positive constants $P_0,P_1>0$ such that if the end pressure
$P_e\in (P_1, P_0)$, there exists a unique transonic shock solution $(\bar{u}^-,
\bar{\rho}^-)$ and $(\bar{u}^{+},\bar{\rho}^{+})$ defined on $I^-=[L_0,L_s)$ and
$I^+=(L_s, L_1)$ respectively, with a shock located at $x_1=L_{s}\in (L_0,L_1)$. In
addition, the shock position $x_1=L_s$ increases as the exit pressure $P_{e}$
decreases. Furthermore, the shock position $L_s$ approaches to $L_1$ if $P_{e}$ goes
to $P_1$ and $L_s$ tends to $L_0$ if $P_{e}$ goes to $P_0$.
\end{lemma}

\begin{proof}
	
	The existence and uniqueness of smooth supersonic flow $(\bar{u}^-,
	\bar{\rho}^-)$ starting from $(\rho_0,u_0)$ on $[L_0,L_1]$ is trivial. Suppose
	the shock occurs at $x_1=L_s\in (L_0,L_1)$, then it is well-known that there
	exists a unique subsonic state $(\bar{u}^+(L_s), \bar{\rho}^+(L_s))$ satisfying
	the Rankine-Hugoniot conditions \eqref{rh} and the entropy condition. With
	$(\bar{u}^+(L_s), \bar{\rho}^+(L_s))$ as the initial data, the equation
	\eqref{1df} has a unique smooth solution $(\bar{u}^+, \bar{\rho}^+)$ on
	$[L_s,L_1]$. Denote $p_e= (\bar{\rho}^+(L_1))^{\gamma}$. In the following, we
	show that the monotonicity between the shock position $x_1=L_s$ and the exit
	pressure $P_e=(\bar{\rho}^+(L_1))^{\gamma}$. $\bar{\rho}^+(L_1)$ is regarded as
	a function of $L_s$.
	Since $(\bar{\rho}^+ \bar{u}^+)(L_s)=(\bar{\rho}^- \bar{u}^-)(L_s)=J=\rho_0
	u_0>0$, then
	\be\label{1df12}
	\bar{u}^-(L_s)+ \frac{ J^{\gamma-1}}{(\bar{u}^-(L_s))^{\gamma}}=\bar{u}^+(L_s)+
	\frac{ J^{\gamma-1}}{(\bar{u}^+(L_s))^{\gamma}}.
	\ee
	It follows from the second equation in \eqref{1df} that
	\be\no
	\frac12
	(\bar{u}^+(L_1))^2+\frac{\gamma}{\gamma-1}({\bar\rho}^+(L_1))^{\gamma-1}-
	\bar\Phi(L_1)=
	 \frac12
	(\bar{u}^+(L_s))^2+\frac{\gamma}{\gamma-1}(\bar{\rho}^+(L_s))^{\gamma-1}-
	\bar\Phi(L_s).
	\ee
	Differentiating with respect to $L_s$, one deduces that
	\be\label{ldf13}
	&&\left(\gamma(\bar{\rho}^+(L_1))^{\gamma-2}-\frac{J^2}{(\bar{\rho}^+(L_1))^3}\right)\frac{d
	 \bar{\rho}^+(L_1)}{d L_s}\\\no
	&&=\left(\gamma(\bar{\rho}^+(L_s))^{\gamma-2}-\frac{J^2}{(\bar{\rho}^+(L_s))^3}\right)\frac{d
	 \bar{\rho}^+(L_s)}{d L_s}-\bar{f}(L_s)=:I.
	\ee
	Also \eqref{1df12} yields that
	\be\no
	\left\{1-\frac{\gamma J^{\gamma-1}}{(\bar{u}^+(L_s))^{\gamma+1}}\right\}\frac{d
	\bar{u}^+(L_s)}{d L_s}=\left\{1-\frac{\gamma
	J^{\gamma-1}}{(\bar{u}^-(L_s))^{\gamma+1}}\right\}\frac{d \bar{u}^-(L_s)}{d
	L_s}=\frac{\bar{f}(L_s)}{\bar{u}^-(L_s)}.
	\ee
	Finally, we conclude that
	\be\no
	I&=&-\left\{\gamma
	(\rho^+(L_s))^{\gamma-1}-\frac{J^2}{(\rho^+(L_s))^2}\right\}\frac{1}{\bar{u}^+(L_s)}\frac{d
	 \bar{u}^+(L_s)}{d L_s}-\bar{f}(L_s)\\\no
	&=&\frac{\bar{f}(L_s)(\bar{u}^+(L_s)-\bar{u}^-(L_s))}{\bar{u}^-(L_s)}<0.
	\ee
	Since the coefficients
	\be\no
	\gamma(\bar{\rho}^+(L_1))^{\gamma-2}-\frac{J^2}{(\bar{\rho}^+(L_1))^3}>0,
	\ee
	then (\ref{ldf13}) implies that the end density $\bar\rho^+(L_1)$ is a strictly
	decreasing function of the shock position $x_1=L_s$. It follows that the end
	pressure $P_e=(\bar\rho^+(L_1))^{\gamma}$ is a strictly decreasing and
	continuous differentiable function on the shock position $x_1=L_s$. In
	particular, when $L_s=L_0$ and $L_s=L_1$, there are two different end pressure
	$P_1,P_2$ with $P_0>P_1$. Hence, by the monotonicity one can obtain a transonic
	shock for the end pressure $P_e\in(P_1,P_0)$.
\end{proof}
\begin{remark}
{\it Lemma \ref{1dtransonic-shock} shows that the external force helps to stabilize the transonic shock in flat nozzles and the shock position is uniquely determined.
}
\end{remark}

The one dimensional transonic shock solution $(\bar{u}^{\pm}, \bar{\rho}^{\pm})$ with a shock occurring at $x_1= L_s$ constructed in Lemma \ref{1dtransonic-shock} will be called the background solution in this paper. The extension of the subsonic flow $(\bar
u^+(x_1),\bar\rho^+(x_1))$ of the background solution to $L_s-\delta_0<x_1<L_1$ for a small positive number $\delta_0$ will be denoted by
$(\hat u^+(x_1),\hat\rho^+(x_1))$.

\begin{figure}[h]
	\centering
	\begin{tikzpicture}
		\draw[->,thick] (-1,0)  -- (7.5,0) node (xaxis) [right] {$x_1$};
		\draw[->,thick] (0,-2.5) -- (0,2.5) node (yaxis) [above] {$x_2$};
		
		\draw[-] (1,1.5) -- (6,1.5);
		\draw[-] (1,-1.5) -- (6,-1.5);
		\draw[-] (1,1.5) -- (1,-1.5) node [left]{$L_0$};
		\draw[-] (6,1.5) -- (6,-1.5)node [right]{$L_1$};
		
		\draw[-] (3.5,1.5) -- (3.5,-1.5) node [below]{$x_1=L_s$};
		
		\draw[-] (6,1.5) -- (6,0.7) node [right]{$P=P_e+\epsilon P_{ex}$};
		
		\draw[->] (1.5,0.8)  node [above]{$\quad\quad\quad\quad supesonic$}--
		(3,0.8);
		\draw[->] (1.5,-0.8) node [below]{$\quad\quad\quad\quad \Omega^-$}--
		(3,-0.8);
		
		\draw[->] (4,0.8) node [above]{$\quad\quad\quad\quad subsonic$}-- (5.5,0.8);
		\draw[->] (4,-0.8) node [below]{$\quad\quad\quad\quad \Omega^+$} --
		(5.5,-0.8);
		
		\draw[->] (3,0.27) node [left]{$shock$} -- (3.4,0.27) node [right]
		{$\,\,
			x_1=\xi(x_2)$};
		
		\draw[domain=2:5,smooth,rotate around ={90:(3.5 ,0)}]
		plot(\x,{(0.08)*sin(21.5*\x r)});	
	\end{tikzpicture}
	\caption{Nozzle}
\end{figure}

It is natural to focus on the
structural stability of this transonic shock flows. For simplicity, we only investigate the structural stability under
suitable small perturbations of the end pressure. Therefore, the supersonic incoming flow is unchanged and remains to be $(\bar{u}^-(x_1), 0, \bar{\rho}^-(x_1))$.

\par Assume that the possible shock curve $\Sigma$ and the flow behind the shock are denoted by $x_1=\xi(x_2)$ and $(u_1^+,u_2^+,P^+)(x)$ respectively (See Figure 1). Let $\Omega^+=\{(x_1,x_2): \xi(x_2)<x_1<L_1, -1<x_2<1\}$  denotes the subsonic region of the flow. Then the Rankine-Hugoniot conditions on $\Sigma$ gives
\be\label{RH-condi}
\begin{cases}
[\rho u_1]- \xi'(x_2) [\rho u_2]=0,\\
[\rho u_1^2+P]- \xi'(x_2) [\rho u_1 u_2]=0,\\
[\rho u_1 u_2]- \xi'(x_2) [\rho u_2^2+P]=0.
\end{cases}
\ee
\par In addition, the pressure $P$ satisfies the physical entropy conditions
\be\label{entr-condi}
P^+(x)>P^-(x) \quad\text{on}\,\Sigma.
\ee
\par Since the flow is tangent to the nozzle walls $x_2=\pm 1$, then
\be\label{u2-BC}
u_2^+(x_1,\pm 1)=0.
\ee
\par The end pressure is perturbed by
\be\label{pres-condi}
P^+(L_1,x_2)=P_e+\epsilon P_{ex}(x_2),
\ee
due to some technical reasons, we may readily suppose that $P_{ex}(x_2)=P_e^{\frac{1}{\gamma}}\hat P_{ex}(x_2)\in C^{2,\alpha}([-1,1]) (\alpha\in (0,1))$ satisfies the compatibility conditions
\be\label{comp-pressure}
\hat P_{ex}'(\pm 1)=0.
\ee

\par
The following theorem gives the main results of this paper.
\begin{theorem}\label{thm1}
Under the assumptions on the external force and the exit pressure, there exists a constant $\epsilon_0>0$ such that for all $\epsilon\in (0,\epsilon_0]$, the system \eqref{2deuler-force},\eqref{RH-condi}-\eqref{pres-condi} has a unique transonic shock solution $(u_1^+(x),u_2^+(x),P^+(x);\xi(x_2))$ which admits the following properties:\\
(i).\,The shock $x_1=\xi(x_2)\in C^{3,\al}([-1,1])$, and satisfies
\be\label{copar_shoc}
\|\xi(x_2)-L_s\|_{C^{3,\al}([-1,1])}\leq C\epsilon,
\ee
where the positive constant $C$ only depends on the background solution, the exit pressure and $\al$.
 \\
(ii).\,The velocity and pressure in subsonic region $(u_1^+,u_2^+,P^+)(x)\in C^{2,\alpha}(\bar\Omega^+)$, and there holds
\be\label{copar_u1u2p}
\|(u_1^+,u_2^+,P^+)(x)-(\hat u,0,\hat P)\|_{C^{2,\al}(\bar\Omega^+)}\leq C\epsilon,
\ee
where $\Omega^+=\{(x_1,x_2): \xi(x_2)<x_1<L_1, -1<x_2<1\}$ is the subsonic region and
$(\hat u,0,\hat P)=(\hat u(x_1),0,P(\hat\rho(x_1)))$ is the extended background
solution.
\end{theorem}

Our proof is influenced by the approach developed in
\cite{LXY2009a,LXY2009b,LXY2013}, yet the reformulation of the problem is different from there. It is well-known that steady Euler equations are hyperbolic-elliptic coupled in subsonic region. The entropy and Bernoulli's function are conserved along the particle path, while the pressure and the flow angle satisfy a first order elliptic system in subsonic region. These facts are widely used in the structural stability analysis for the transonic shock problems in flat or divergent nozzles, one may refer to \cite{C2005,CCS2006,ccf07,LXY2009a,LXY2009b,LXY2013,s95,XY2008full,Y2006} and the reference therein. Here we resort to a different decomposition based on the deformation and curl of the velocity developed in \cite{WengX2019,Weng2019} for three dimensional steady Euler and Euler-Poisson systems. The idea in that decomposition is to rewrite the density equation as a Frobenius inner product of a symmetric matrix and the deformation matrix by using the Bernoulli's law. The vorticity is resolved by an algebraic equation of the Bernoulli's function and the entropy. We should mention that there are several different decompositions to three dimensional steady Euler system \cite{cx14,C2008,cy08,lxy16,w15,XY2008full} developed by many researchers for different purposes. An interesting issue that deserves further discussion is when using the deformation-curl decomposition to deal with the transonic shock problem, the end pressure boundary condition becomes nonlocal since it involves the information from the shock front (see \eqref{v1v2_BC_lin}). However, this nonlocal boundary condition reduces to be local after introducing the potential function (see \eqref{phi_BC}).

\par The rest of this paper will be organized as follows.  In Section 2, we
reformulate the original 2-D problem  (\ref{2deuler-force})-(\ref{pres-condi}) by
deformation-curl decomposition developed in \cite{WengX2019,Weng2019} so that
one can rewrite the system (\ref{2deuler-force}) with the velocity and the Bernoulli function. We
obtain a $2\times 2$  first order system for the velocity field, a transport type
equation for the  Bernoulli function and the first order ordinary differential
equation for the shock after linearization. In Section 3, we design an elaborate
iteration scheme inspired by the works \cite{LXY2009b} for the nonlinear system. The
investigation of well-posedness and regularity for the linear system are given in
the reminder part of this section. In section 4, we prove the main existence and uniqueness theorem.

\section{Reformulation of the problem }

Different from previous works on transonic shock problems \cite{C2005,CCS2006,ccf07,LXY2009a,LXY2009b,LXY2013}, we will use the deformation-curl decomposition developed in \cite{WengX2019,Weng2019} for steady Euler system to decompose the original system (\ref{2deuler-force}) into an equivalent system (\ref{def-curl}), where the hyperbolic quantity $B$ and elliptic quantities $u_1,u_2$ are effectively decoupled in subsonic regions. To this end, define the Bernoulli's function
\be\label{Bernou}
B=\frac{1}{2} |{\bf u}|^2 + h(\rho)-\Phi,
\ee
where $h(s)=\frac{\gamma}{\gamma-1}s^{\gamma-1}$ is the enthalpy function. Hence the density can be expressed by the Bernoulli function and velocity field as
\be\label{dens}
\rho=H(B,\Phi,|{\bf u}|^2)=\left[\frac{\gamma-1}{\gamma}(B+\Phi-\frac{1}{2}|{\bf u}|^2)\right]^{\frac{1}{\gamma-1}}.
\ee
Consequently, the 2-D Euler system (\ref{2deuler-force}) with unknown function $(u_1,u_2,P)$ is equivalent to the following system
\be\label{def-curl}
\begin{cases}
\displaystyle \sum_{i,j=1}^2(c^2(H)\delta_{ij}- u_i u_j) \p_i u_j + u_1\p_1 \Phi + u_2 \p_2 \Phi=0,\\
\p_1 u_2 -\p_2 u_1 =-\frac{\p_2 B}{u_1},\\
u_1 \p_1 B + u_2 \p_2 B=0,
\end{cases}
\ee
with unknown function $(u_1,u_2,B)$.
\par The shock curve is determined by
\be\label{shock0}
\xi'(x_2)= \frac{[\rho u_1 u_2]}{[\rho u_2^2+P]}(\xi(x_2),x_2),\quad x_2\in(-1,1).
\ee
Furthermore, it follows from the R-H conditions (\ref{RH-condi}) that
\be\label{rh2}
\begin{cases}
[\rho u_1]= \frac{[\rho u_2] [\rho u_1 u_2]}{[\rho u_2^2 +P]},\\
[\rho u_1^2+ P(\rho)]= \frac{([\rho u_1 u_2])^2}{[\rho u_2^2 +P]}.
\end{cases}
\ee
A direct computation by using (\ref{rh2}) shows that on $x_1=\xi(x_2)$
\be\label{expr_g1g2}
	&&(\rho^+(\xi(x_2),x_2)- \bar{\rho}^+(L_s),u_1^+(\xi(x_2),x_2)-
	\bar{u}^+(L_s))= \\
	&&(h_1,h_2)(\rho^-(\xi(x_2))-\bar\rho^-(L_s), u^-(\xi(x_2))-\bar
	u^-(L_s),(u_2^+(\xi(x_2),x_2))^2) \no
\ee
here $h_i(0,0,0)=0$ for $i=1,2$. In addition, we have
\begin{equation}\label{Der_g1}
	\begin{cases}
\begin{aligned}
\frac{\p h_1}{\p(\rho^--\bar\rho^-)}|_{(0,0,0)}&=2\bar u^-(L_s)\frac{\bar u^+-\bar
u^-}{(\bar u^+(L_s))^2-c^2(\bar\rho^+(L_s))}+1,\\
\frac{\p h_1}{\p(u^--\bar u^-)}|_{(0,0,0)}&=2 \bar\rho^-(L_s)\frac{\bar u^+-\bar
	u^-}{(\bar u^+(L_s))^2-c^2(\bar\rho^+(L_s))}, \\
\frac{\p
h_1}{\p(u_2^+)^2}|_{(0,0,0)}&=\frac{(\bar\rho^+(L_s)\bar
u^+(L_s))^2}{\bar P^+(L_s)-\bar P^-(L_s)}\frac{1}{(\bar
u^+(L_s))^2-c^2(\bar\rho^+(L_s))},
\end{aligned}
\end{cases}
\end{equation}
and
\begin{equation}\label{Der_g2}
	\begin{cases}
	\begin{aligned}
\frac{\p h_2}{\p(\rho^--\bar\rho^-)}|_{(0,0,0)}&=-(\gamma-1)\frac{\bar P^+(L_s)-\bar
P^-(L_s)}{(\bar\rho^+(L_s))^2\bar u^+(L_s)}\frac{\bar u^+\bar
	u^-}{(\bar u^+(L_s))^2-c^2(\bar\rho^+(L_s))},\\
\frac{\p h_2}{\p(u^--\bar u^-)}|_{(0,0,0)}&=\frac{2 \bar\rho^-(L_s)\bar
u^+(L_s)}{\bar\rho^+(L_s)}\frac{\bar u^-\bar u^+}{(\bar
u^+(L_s))^2-c^2(\bar\rho^+(L_s))}+\frac{\bar\rho^-(L_s)}{\bar\rho^+(L_s)}, \\
\frac{\p h_2}{\p(u_2^+)^2}|_{(0,0,0)}&=\frac{\bar\rho^+(L_s)\bar
	u^+(L_s)}{\bar P^+(L_s)-\bar P^-(L_s)}\frac{c^2(\bar\rho^+(L_s))}{(\bar
	u^+(L_s))^2-c^2(\bar\rho^+(L_s))}.
\end{aligned}
\end{cases}
\end{equation}
By substituting (\ref{expr_g1g2}) into (\ref{Bernou}), we conclude that there is a
function $h_3$ such that
\be\label{expr_g3}
&&B^+(\xi(x_2),x_2)- \bar{B}^+(L_s)= h_3(\rho^-(\xi(x_2))-\bar\rho^-(L_s), u^-(\xi(x_2))-\bar
u^-(L_s),(u_2^+(\xi(x_2),x_2))^2).
\ee
Thus, Theorem \ref{thm1} is established as long as we solve the problem (\ref{def-curl})-(\ref{shock0}) with boundary conditions \eqref{rh2}, \eqref{u2-BC}-\eqref{pres-condi}. In order to deal with the free boundary value problem (\ref{def-curl})-(\ref{shock0}), we introduce the following transformation to reduce it into a fixed boundary value problem. Setting
\be\label{chag-var-fix-bc}
y_1=\frac{x_1-\xi(x_2)}{L_1-\xi(x_2)}(L_1-L_s) +L_s,\ \ \ \ y_2=x_2,
\ee
then, the domain $\Omega^+=\{(x_1,x_2): \xi(x_2)<x_1<L_1, -1<x_2<1\}$ is changed into
\be\label{chag_inv_var}
Q=\{(y_1,y_2): L_s<y_1<L_1,\ -1<y_2<1\}.
\ee
The inverse change  variable gives
\be\label{chag_inve}
x_1=\xi(y_2)+\frac{L_1-\xi(y_2)}{L_1-L_s}(y_1-L_s)= y_1+\frac{L_1-y_1}{L_1-L_s}(\xi(y_2)-L_s),\ x_2=y_2.\no
\ee
We now set for $y\in Q$
\be\no
(\tilde u_j,\tilde\rho,\tilde B,\tilde\Phi)(y_1,y_2)=( u_j,\rho, B,\Phi)\left(\frac{L_1-\xi(y_2)}{L_1-L_s}(y_1-L_s)+\xi(y_2), y_2\right),\,\,j=1,2.
\ee
The shock equation (\ref{shock0}) becomes to
\begin{equation}\label{shock1}
\begin{aligned}
\xi'(y_2)&= \frac{(\rho  u_1  u_2)(\xi(y_2),y_2)}{P^+(\rho)(\xi(y_2),y_2)-P^-(\xi(y_2))+\rho (u_2)^2(\xi(y_2),y_2)},\\
&=\frac{(\tilde\rho  \tilde u_1  \tilde
u_2)(L_s,y_2)}{P^+(\tilde\rho)(L_s,y_2)-P^-(\xi(y_2))+\tilde\rho (\tilde
u_2)^2(L_s,y_2)},\, y_2\in(-1,1),
\end{aligned}
\end{equation}
and the system (\ref{def-curl}) is changed into
\be\label{tilde_v1v2_fin}
&&(c^2(\tilde\rho)-\tilde u_1^2)\frac{L_1-L_s}{L_1-\xi(y_2)} \p_{y_1}\tilde
u_1+c^2(\tilde\rho)\p_{y_2}\tilde u_2+\tilde u_1\frac{L_1-L_s}{L_1-\xi(y_2)}
\p_{y_1}\tilde \Phi =F_1(\tilde{{\bf u}}, \tilde{B}),\\\nonumber
&&\frac{L_1-L_s}{L_1-\xi(y_2)} \p_{y_1}\tilde u_2-\p_{y_2}\tilde u_1-\frac{y_1-L_1}{L_1-\xi(y_2)}\xi'(y_2)\p_{y_1}\tilde u_1+
\frac{\p_{y_2}\tilde B}{\tilde u_1}=F_2(\tilde \u,\tilde B),\\\label{B_eq}
&&\tilde u_1\frac{L_1-L_s}{L_1-\xi(y_2)} \p_{y_1}\tilde B+\tilde u_2\p_{y_2}\tilde
B+\frac{y_1-L_1}{L_1-\xi(y_2)}\xi'(y_2)\tilde u_2\p_{y_1}\tilde B=0,
\ee
where
\be\no
&&F_1(\tilde{{\bf u}}, \tilde{B})=
\tilde u_2^2\p_{y_2}\tilde u_2-(c^2(\tilde\rho)-\tilde
u_2^2)\frac{y_1-L_1}{L_1-\xi(y_2)}\xi'(y_2)\p_{y_1}\tilde u_2
+\tilde u_1\tilde u_2\frac{L_1-L_s}{L_1-\xi(y_2)} \p_{y_1}\tilde u_2\\ \no
&&+\tilde u_2\tilde u_1(\p_{y_2}\tilde u_1+\frac{y_1-L_1}{L_1-\xi(y_2)}\xi'(y_2)\p_{y_1}\tilde u_1)
-\tilde u_2(\p_{y_2}\tilde \Phi+\frac{y_1-L_1}{L_1-\xi(y_2)}\xi'(y_2)\p_{y_1}\tilde \Phi),\\
&&F_2(\tilde{{\bf u}}, \tilde{B})=-{\frac{y_1-L_1}{L_1-\xi(y_2)}\frac{\xi'(y_2)}{\tilde u_1} \p_{y_1}\tilde B}.\no
\ee
Consider the perturbed functions $v_i(y_1,y_2)$, $i=1,2,3,4,$ as
\be\no
&&v_1(y_1,y_2)=\tilde{u}_1(y_1,y_2)-\bar{u}^+(y_1),\ \ v_2(y_1,y_2)= \tilde{u}_2(y_1,y_2),\\\no
&&v_3(y_1,y_2)=\tilde{B}(y_1,y_2)-\bar{B}^+,\ \ \ v_4(y_2)=\xi(y_2)-L_s,
\ee
and define the vector functions
\be
V(y_1,y_2)=(v_1(y_1,y_2),v_2(y_1,y_2),v_3(y_1,y_2),v_4(y_2)).
\ee
It follows from (\ref{shock1}) that the shock satisfies
\be\label{shock2}
v_4'(y_2)=\frac{\tilde\rho (\bar{u} +v_1) v_2(L_s, y_2)}{P^+(\tilde\rho)(L_s,y_2)-P^-(\xi(y_2))+\tilde\rho (\tilde u_2)^2(L_s,y_2)}.
\ee
Through a direct computation, one can derive from  (\ref{expr_g3}) and (\ref{B_eq})
that the Bernoulli function satisfy a transport type equation
\be\label{bernou1}
\begin{cases}
[(\bar u^+ + v_1)(L_1-L_s)+v_2(y_1-L_1)v_4'(y_2)]\p_{y_1} v_3+v_2(L_1- v_4-L_s)\p_{y_2} v_3=0,\\
 v_3(L_s,y_2)= b_3 v_4(y_2) + R_3(y_2),
\end{cases}
\ee
where
\be\label{b3}
b_3=\frac{\bar\rho^-(L_s)-\bar\rho^+(L_s)}{\bar\rho^+(L_s)}\bar f(L_s),
\ee
and $R_3(y_2)=R_3(V(L_s,y_2))=O(|V(L_s,y_2)|)^2$ is an error term of second order.
 We may readily drop  superscribe $+$ on the background solutions if there is no
 risk of confusing. And the first order system for $v_1, v_2$ is given by,
\be\label{v1v2}
\begin{cases}
  (c^2(\bar\rho^+)-(\bar u^+)^2)\p_{y_1} v_1+c^2(\bar\rho^+)\p_{y_2}v_2+B_1(y_1)v_1\\
  \,\,\quad\quad\quad+B_3(y_1)v_3+B_4(y_1)v_4(y_2)=F_3(V,\nabla V),\\
  \p_{y_1}v_2-\p_{y_2}v_1+\frac{L_1-y_1}{L_1-L_s}\bar u'v_4'+\frac{\p_{y_2}v_3}{\bar
  u}=F_4(V,\nabla V),
\end{cases}
\ee
where $F_3, F_4$ represent the remainder term of second order with respect to $V$
and $\nabla V$, and
\be\no
&&B_1(y_1)=\bar f(y_1)-(\gamma+1)\bar u\bar u'=\frac{\gamma \bar
u^2+c^2(\bar\rho^+)}{c^2(\bar\rho^+)-\bar u^2}\bar f>0,\,\\ \no
&&B_3(y_1)=(\gamma-1)\bar u',\,\\ \no
&&B_4(y_1)=\frac{1}{L_1-L_s}[(\gamma-1)\bar f(y_1)(L_1-y_1)\bar u'-\bar u \bar
f+\bar u\bar f'(y_1)(L_1-y_1)].
\ee
It's obvious that the second formula in (\ref{expr_g1g2}) gives the boundary
condition of $v_1$ on the entrance $x_1=L_s$. Meanwhile, the formula (\ref{dens})
after changing variable becomes
\begin{equation}
	\begin{aligned}
		c^2(\tilde\rho)(y)&=\gamma\tilde\rho^{\gamma-1}=(\gamma-1)(\tilde
		B-\frac{1}{2}|\tilde{\bf u}|^2+\bar \Phi)\\
		&=(\gamma-1)(\bar B^++v_3-\frac{1}{2}(\bar
		u^++v_1)^2-\frac{1}{2}v_2^2+\bar \Phi)\\
		&=c^2(\bar\rho^+)+(\gamma-1)(v_3-\bar
		u^+v_1-\frac{1}{2}v_1^2-\frac{1}{2}v_2^2),
	\end{aligned}
\end{equation}
which together with (\ref{pres-condi}) gives the boundary condition of $v_1$ on the
exit $x_1=L_1$. Hence, the boundary conditions to the system (\ref{v1v2}) read as
follow
\be\label{v1v2_BC}
\begin{cases}
  v_1(L_s,y_2)=b_2v_4(y_2)+R_2(y_2),\\
  v_1(L_1,y_2)=\frac{1}{\bar u(L_1)}(v_3(L_1,y_2)-\epsilon \hat P_{ex}(y_2))+R_4(y_2),\\
v_2(y_1,\pm 1)=0,
\end{cases}
\ee
here
\be\label{b2}
b_2=\frac{\bar{\rho}^-(L_s)\bar{u}^+(L_s)}{\bar{\rho}^+(L_s)
	[(\bar{u}^+(L_s))^2-c^2(\bar{\rho}^+(L_s))]}\bar{f}(L_s)<0,
\ee
and
$R_2(y_2)=R_2(V(L_s,y_2))=O(|V(L_s,y_2)|)^2$,\,$R_4(y_2)=R_4(V(L_1,y_2))=O(|V(L_1,y_2)|)^2$
are error terms of second order.
Based on our reformulation, Theorem \ref{thm1} follows from the following results.
\begin{theorem}\label{thm3.1}
Under the same assumptions as in Theorem \ref{thm1} there exists a positive constant
$\epsilon_0>0$ such that for each $\epsilon \in (0,\epsilon_0]$, the system
(\ref{shock2})-(\ref{v1v2_BC}) has a unique solution $V\in (C^{2,\alpha}(\bar{Q}))^3\times C^{3,\alpha}([-1,1])$ satisfying the following
estimate
\be\label{est_th3.1}
\sum_{i=1}^{3}\| v_i\|_{C^{2,\al}(\bar Q)}+\| v_4\|_{C^{3,\al}([-1,1])}\leq C\epsilon,
\ee
where the constant $C$ depends only on the background solution, the exit pressure and $\al\in (0,1)$.
\end{theorem}
\section{iteration scheme and  the linear system}
In the first part of this section, we construct an iteration scheme for the
nonlinear system, and the
problem is reduced to the solvability of corresponding linear systems. Indeed, it
turns out that the linear system is a non-local elliptic equation of second order
with a free parameter denoting the relative location of the shock position on the
wall $x_2=-1$. Then we study the existence, uniqueness and regularity for this linear
system in the remainder part of this section.
\subsection{Iteration Scheme}\label{iteration}
\par Inspired by \cite{LXY2009b}, we will develop an iteration scheme to prove Theroem \ref{thm3.1}. Consider the Banach space
\be\label{itr-spac}
\V_{\delta}:=\{V:\displaystyle\sum_{i=1}^{3}\|v_i\|_{C^{2,\al}(\bar Q)}+|v_4|_{C^{2,\al}[-1,1]}\leq \delta;\p_{y_2}v_{j}(y_1,\pm 1)=0,\\
j=1,3;v_2(y_1,\pm 1)=\p_{y_2}^2v_2(y_1,\pm 1)=0; v_4'(\pm 1)=v_4^{(3)}(\pm 1)=0\},\no
\ee
here $\delta>0$ is a small constant to be determined later. For a fix $\hat V\in \V_{\delta}$, equivalently, we have the following quantity
\be\no
(\hat v_1,\hat v_2,\hat B, \hat\rho, \hat P,\hat\xi)(y).
\ee
We now define the linearized scheme to the problem (\ref{shock2})-(\ref{v1v2_BC}) as
follows.
\par Firstly, thanks to (\ref{shock2}), $v_4$ is determined by
\be\label{shock3}
v_4'(y_2)=b_0 v_2(L_s,y_2) +F_5(\hat V)(L_s,y_2),
\ee
where
\be\no
&&b_0=\frac{(\bar{\rho} \bar{u})(L_s)}{P^+(\bar\rho)(L_s)-P^-(L_s)}>0,\\ \no
&&F_5(y_2)=\{\frac{\hat\rho \hat u_1}{P^+(\hat\rho)(L_s,y_2)-P^-(\hat\xi(y_2))+\hat\rho (\hat u_2)^2(L_s,y_2)} -b_0\}v_2(L_s,y_2), \no
\ee
hence, one can express the shock as
\be\label{shoc_expr}
v_4(y_2)=v_4(-1)+\int_{-1}^{y_2}b_0v_2(L_s,\tau) d\tau + R_5(y_2),
\ee
where  $R_5(y_2)=\int_{-1}^{y_2}F_5(\tau) d\tau$ is a error term of second order. Due to $\hat V\in\V_{\delta}$, we have
\be\label{err_F5}
F_5(\pm 1)=F_5^{''}(\pm 1)=0,\,\|F_5\|_{C^{k,\al}[-1,1]}\leq C\delta\|\hat v_2\|_{C^{k,\al}(\bar Q)},\,k=0,1,2.
\ee
\par Secondly, Using (\ref{bernou1}), we get the linearized transport equation for $v_3$:
\be\no
[(\bar u + \hat v_1)(L_1-L_s)+\hat v_2(y_1-L_1)\hat v_4'(y_2)]\p_{y_1} v_3+\hat v_2(L_1- \hat v_4-L_s)\p_{y_2} v_3=0\,\,\text{in}\,Q,
\ee
with initial data
\be\no
v_3(L_s,y_2)= b_3 v_4(y_2) + R_3(y_2).\no
\ee
Thus, it can be solved by characteristic methods. Let $y_2(s;\beta)$ be the characteristics going through $(y_1, y_2)$ with $y_2(L_s)=\beta$, i.e.
\be\label{eq_char}
\begin{cases}
\frac{d y_2}{ds}(s;\beta)=\frac{\hat v_2(L_1- \hat v_4-L_s)}{(\bar u + \hat v_1)(L_1-L_s)+\hat v_2(y_1-L_1)\hat v_4'(y_2)},\ \ L_s<s<L_1,\\
y_2(L_s)=\beta.
\end{cases}
\ee
It is noted that $\beta$ can be also regarded as a function of $y=(y_1,y_2)$, this is denoted by $\beta=\beta(y)$, which leads to
\be\label{v3_expr}
v_3(y_1,y_2)=v_3(L_s,\beta(y))=b_3 v_4(y_2)+ F_6(y),
\ee
where
\be\no
F_6(y)=b_3\int_{y_2}^{\beta(y)} v_4'(\tau) d\tau +R_3(\hat V(L_s,\beta(y)))
\ee
is an error term of second order. Furthermore, we have
\be\label{err_F6}
&&\p_{y_2}F_6(y_1,\pm 1)=0,\\
&&\|F_6\|_{C^{k,\al}(\bar Q)}\leq C\delta (\sum_{i=1}^{3}\|\hat v_i\|_{C^{k,\al}(\bar Q)}+\|\hat v_4\|_{C^{k+1,\al}(\bar Q)}),\,k=0,1,2.\no
\ee
\par It remains to determine the velocity $v_1,v_2$ and the shock position difference $v_4(-1)$ on the wall $x_2=-1$. Substituting (\ref{shoc_expr}) and (\ref{v3_expr}) into (\ref{v1v2}) and (\ref{v1v2_BC}), we get the following linearized system for $v_1,v_2$ with a unknown parameter $v_4(-1)$:
\be\label{v1v2-lin}
\begin{cases}
  \p_{y_1} v_1+\frac{1}{1-\bar M^2}\p_{y_2}v_2+\frac{B_1(y_1)}{c^2(\bar\rho^+)-(\bar u^+)^2}v_1\\
  \quad\quad+\frac{B_3(y_1)b_3+B_4(y_1)}{c^2(\bar\rho^+)-(\bar u^+)^2} (v_4(-1)+\int_{-1}^{y_2}b_0v_2(L_s,\tau) d\tau)
  =G_1(y),\\
  \p_{y_1}v_2-\p_{y_2}[v_1-\la(y_1) (v_4(-1)+\int_{-1}^{y_2}b_0v_2(L_s,\tau) d\tau)] =G_2(y),
\end{cases}
\ee
and the boundary conditions
\be\label{v1v2_BC_lin}
\begin{cases}
  v_1(L_s,y_2)=b_2(v_4(-1)+\int_{-1}^{y_2}b_0v_2(L_s,\tau) d\tau)+R_6(y_2),\\
  v_1(L_1,y_2)=\frac{b_3}{\bar u(L_1)}(v_4(-1)+\int_{-1}^{y_2}b_0v_2(L_s,\tau) d\tau)- \frac{\epsilon \hat P_{ex}(y_2)}{\bar u(L_1)}+R_7(y_2),\\
v_2(y_1,\pm 1)=0,
\end{cases}
\ee
where
\be\no
&&\la(y_1)=\frac{L_1-y_1}{L_1-L_s}\bar u'+\frac{b_3}{\bar u},\\ \no
&&R_6(y_2)=b_2R_5(y_2)+R_2(y_2),\\ \no
&&R_7(y_2)=\frac{b_3R_5(y_2)+F_6(L_1,y_2)}{\bar u(L_1)}+R_4(y_2),\\ \no
&&G_1(y)=\frac{F_3(\hat V,\nabla \hat V)-B_3(y_1)R_5(y_2)-B_3(y_1)F_6(y)-B_4(y_1)R_5(y_2)}{(c^2(\bar\rho^+)-(\bar u^+)^2)},\\ \no
&&G_2(y)=F_4(\hat V,\nabla \hat V)-\frac{\p_{y_2}F_6(y)}{\bar u}+\la(y_1)\p_{y_2}R_5(y_2), \no
\ee
are error terms of second order. It follows from (\ref{err_F5}),(\ref{err_F6}) and a simple calculation that
\be\label{err_G}
&&\p_{y_2}G_1(y_1,\pm 1)=0,\,G_2(y_1,\pm 1)=0, \\
&&\|G_i\|_{C^{k-1,\al}(\bar Q)}\leq C\delta \|\hat V\|_{C^{k,\al}(\bar Q)},\,i,k=1,2,\no
\ee
and
\be\label{err_R}
&&R_6'(\pm 1)=0,\,R_7'(\pm 1)=0,\\
&&\|(R_6,R_7)\|_{C^{k,\al}[-1,1]}\leq C\delta\|\hat V\|_{C^{k,\al}(\bar Q)},\,k=0,1,2.\no
\ee
\par The second equation in (\ref{v1v2-lin}) implies that there is a potential function $\phi(y)$ satisfy
\be\label{phi}
\begin{cases}
  \p_{y_2}\phi=v_2,\,\,\phi(L_s,-1)=0,\\
  \p_{y_1}\phi=v_1-\la(y_1) (v_4(-1)+\int_{-1}^{y_2}b_0v_2(L_s,\tau) d\tau)+\int_{-1}^{y_2}G_2(y_1,\tau)d\tau,
\end{cases}
\ee
it follows that $v_1,v_2$ can be represented by
\be\label{rprstv1v2}
\begin{cases}
  v_2=\p_{y_2}\phi,\\
  v_1=\p_{y_1}\phi+\la(y_1)[v_4(-1)+b_0\phi(L_s,y_2)]-\int_{-1}^{y_2}G_2(y_1,\tau)d\tau.
\end{cases}
\ee
Substituting (\ref{rprstv1v2}) into the first equation in  (\ref{v1v2-lin}),  we
conclude that $\phi$ satisfying the following non-local elliptic equation of second
with an unknown constat $v_4(-1)$
\be\label{phi_eq}
\p_{y_1}^{2}\phi+\frac{1}{1-\bar M^2}\p_{y_2}^{2}\phi+\la_{1}(y_1)\p_{y_1}\phi+\la_{0}(y_1)b_0(\frac{v_4(-1)}{b_0}+\phi(L_s,y_2))\\ \no
=G_1(y)+\la_1(y_1)\int_{-1}^{y_2}G_2(y_1,\tau)d\tau+\p_{y_1}\int_{-1}^{y_2}G_2(y_1,\tau)d\tau,\no
\ee
where
\be\no
&&\la_{1}(y_1)=\frac{B_1(y_1)}{c^2(\bar\rho)-\bar u^2}>0,\\ \no
&&\la_{0}(y_1)=\frac{1}{c^2(\bar\rho)-\bar u^2}[(c^2(\bar\rho)-\bar u^2)\la'+B_1\la+B_3b_3+B_4]. \no
\ee
Similarly, substituting (\ref{rprstv1v2}) into  the boundary conditions
(\ref{v1v2_BC_lin}), combine with the boundary condition of $\phi$ in (\ref{phi}),
we have
\be\label{phi_BC}
\begin{cases}
  \p_{y_1}\phi(L_s,y_2)=b_0(b_2-\la(L_s))(\frac{v_4(-1)}{b_0}+\phi(L_s,y_2))+R_6(y_2)+\int_{-1}^{y_2}G_2(L_s,\tau)d\tau,\\
  \p_{y_1}\phi(L_1,y_2)=- \frac{\epsilon \hat P_{ex}(y_2)}{\bar u(L_1)}+R_7(y_2)+\int_{-1}^{y_2}G_2(L_1,\tau)d\tau,\\
\p_{y_2}\phi(y_1,\pm 1)=\phi(L_s,-1)=0.
\end{cases}
\ee

So far, we have reduced the problem (\ref{v1v2-lin})-(\ref{v1v2_BC_lin}) into a
non-local elliptic equation of $\phi$ with an unknown constant $v_4(-1)$. Hence, it
is sufficient to establish the solvability and regularity results for the problem
(\ref{phi_eq})-(\ref{phi_BC}) in order to study the original problem. We are going
to do it in the next subsection.
\subsection{A non-local elliptic equation with a free constant}
In this section, we prove the existence, uniqueness and regularity results for the
problem (\ref{phi_eq}). To this end, we consider the following more concise form of
second order elliptic system with an unknown constant $\ka$
\be\label{gen_2_ellip}
\begin{cases}
\p_{y_1}^2\phi+a_2(y_1)\p_{y_2}^2\phi+a_1(y_1)\p_{y_1}\phi-a_0(y_1)(\ka+\phi(L_s,y_2))=\p_{y_1}f,\,\text{in}\,Q,\\
\p_{y_1}\phi(L_s,y_2)-a_3(\ka+\phi(L_s,y_2))=g_1(y_2),\\
\p_{y_1}\phi(L_1,y_2)=g_2(y_2),\\
\p_{y_2}\phi(y_1,\pm 1)=\phi(L_s,-1)=0,\\
\end{cases}
\ee
where the smooth coefficients $a_i(y_1)$, i=0,1,2  and the constant $a_3$ satisfy
\be\label{sol_condi}
a_1(y_1)<C_1,\,0<C_0<a_i(y)<C_1,,i=0,2,3,\,
\ee
and the parameter $\ka$ is a constant to be determined with the solution itself.
\par The first lemma  implies that the inhomogeneous problem corresponding to the
system
(\ref{gen_2_ellip}) without the unknown constant
has a unique weak solution.
\begin{lemma}\label{sol_inhomo}
	Suppose that $f\in L^2(Q)$ and $g_i\in L^2(-1,1), i=1,2$, then, there exists a
	suitable large positive constant $K$, such that the  following inhomogeneous
	second order elliptic equation
	\be\label{gen_2_ellip_inhom}
	\begin{cases}
		-\p_{y_1}^2\phi-a_2(y_1)\p_{y_2}^2\phi-a_1(y_1)\p_{y_1}\phi+K\phi+a_0(y_1)\phi(L_s,y_2)=\p_{y_1}f,\,\text{in}\,Q,\\
		\p_{y_1}\phi(L_s,y_2)-a_3\phi(L_s,y_2)=g_1(y_2),\\
		\p_{y_1}\phi(L_1,y_2)=g_2(y_2),\\
		\p_{y_2}\phi(y_1,\pm 1)=0,\\
	\end{cases}
	\ee
	admits a unique weak solution $\phi\in H^1(Q)$ satisfying
	\be\label{H1_est_inhom}
	\|\phi\|_{H^1(Q)}\leq C(\|(g_1,g_2)\|_{L^2(-1,1)}+\|f\|_{L^2(Q)}),
	\ee
	for some positive constant $C>0$.
\end{lemma}
\begin{proof}
	For $\phi,\psi\in H^1(Q)$, define the bilinear form
	\be\label{bilin_for}\no
	\B[\phi,\psi]=\int_{Q}\p_{y_1}\phi\p_{y_2}\psi dy+\int_{Q}
	a_2(y_1)\p_{y_2}\phi\p_{y_2}\psi dy-\int_{Q}a_1(y_1)\psi\p_{y_1}\phi  dy\\\no
	\quad\quad+K\int_{Q}\phi\psi dy+\int_{Q}a_0(y_1)\phi(L_s,y_2)\psi
	dy+a_3\int_{-1}^{1}\phi(L_s,y_2)\psi(L_s,y_2) d y_2,
	\ee
	and the linear functional on $H^1(Q)$
	\be\no
	l(\psi)=\int_{-1}^{1}   g_2(y_2) \psi(L_1,y_2)d y_2-\int_{-1}^{1}
	g_1(y_2)\psi(L_s,y_2)dy_2-\int_{Q}\p_{y_1}f\psi dy.
	\ee
	It's obviously that the linear functional $l(\psi)$ on $H^1(Q)$ is continuous,
	i.e,
	\be\label{funal_est}
	|l(\psi)|\leq
	C(\|(g_1,g_2)\|_{L^2(-1,1)}+\|f\|_{L^2(Q)})\|\psi\|_{H^1(Q)},
	\ee
	where we have used the trace theorem. So, what we need to do is just verify that
	the conditions of Lax-Milgram Theorem is satisfied for the bilinear form $\B$.
	The boundedness of $\B_{K}$ is trivial, we will show that $\B_{K}$ is also
	coercive. Denote $\Lambda=\min\{1,C_0\}>0$, then
	\be\no
	\Lambda\int_{Q}|\nabla\phi|^2 dy+K\int_{Q}|\phi|^2dy\leq
	\B[\phi,\phi]+\int_{Q}a_1(y_1)\phi\p_{y_1}\phi  dy\\
	-\int_{Q}a_0(y_1)\phi(L_s,y_2)\phi(y_1,y_2) dy
	-a_3\int_{-1}^{1} |\phi(L_s,y_2)|^2 dy_2\no,
	\ee
	Cauchy's inequality gives
	\be\no
	\int_{Q}a_1(y_1)\phi\p_{y_1}\phi  dy\leq C_1\epsilon \int_{Q}|\nabla\phi|^2 dy
	+\frac{C_1}{4\epsilon}\int_{Q}|\phi|^2dy,
	\ee
	and
	\be\no
	\int_{Q}a_0(y_1)\phi(L_s,y_2)\phi(y_1,y_2) dy\leq C_{tr}(L_1-L_s)C_1\epsilon
	\int_{Q}|\nabla\phi|^2 dy +\frac{C_1}{4\epsilon}\int_{Q}|\phi|^2dy,
	\ee
	Then, fix $\epsilon_0$ such that
	$C_1\epsilon_0(1+(L_1-L_s)C_{tr})<\frac{\Lambda}{2}$, and choosing
	$K=\max\{\Lambda,\frac{C_1}{\epsilon_0}\}$, thanks to the positivity of $a_3$,
	we obtain
	\be\no
	\B[\phi,\phi]\geq\frac{\Lambda}{2}\|\phi\|_{H^1(Q)},
	\ee
	the unique existence follows from the Lax-Milgram Theorem and (\ref{funal_est})
	gives the estimates (\ref{H1_est_inhom}). Thus the proof is completed.
\end{proof}
\par The unique existence of regular solution to the non-local system
(\ref{gen_2_ellip})
is established in the following proposition.
\begin{proposition}\label{uniq_exis}
For any $f\in C^{1,\al}(\bar Q)$, $g_i\in C^{\al}(\bar Q)$, there is a unique weak
solution $(\phi,\kappa)$, such that $\phi\in H^1(Q)$ and the following estimate holds
\be\label{H1_est}
\|\phi\|_{H^1(Q)}+|\ka|\leq C(\|f\|_{C^{\al}(\bar Q)}+|(g_1,g_2)|_{C^{1,\al}[-1,1]}).
\ee
Moreover, if the compatibility conditions
\be\label{fg_condi}
\p_{y_2}f(y_1,-1)=\p_{y_2}f(y_1,1)=0,\,g_i'(-1)=g_i'(1)=0,\,i=1,2,
\ee
are fulfilled, then $\phi\in C^{2,\alpha}(\bar Q)$
\be\label{c1_al_phi}
\|\phi\|_{C^{1,\al}(\bar Q)}\leq C(\|f\|_{C^{\al}(\bar
	Q)}+|(g_1,g_2)|_{C^{\al}[-1,1]}+\|\phi\|_{H^1(Q)}+|\ka|),
\ee
and
\be\label{c2_al_phi}
\|\phi\|_{C^{2,\al}(\bar Q)}\leq C(\|f\|_{C^{1,\al}(\bar
	Q)}+|(g_1,g_2)|_{C^{1,\al}[-1,1]}+\|\phi\|_{H^1(Q)}+|\ka|).
\ee
for some positive constant $C>0$.
\end{proposition}
\begin{proof}
The proof is divided into two steps.\\
Step 1: Regularity of  weak solutions. we will use the symmetric extension methods
to
exclude the possible singularities that may appear at the corner, which implies
that the weak solution $\phi\in H^1(Q) $ to the system (\ref{gen_2_ellip}) are
essentially more regular.
To this end, introduce the notation
\be\no
Q^{*}:=\{(y_1,y_2): L_s<y_1<L_1,\ -2<y_2<2\}.
\ee
and define the extended function $\phi^*(y)$ on $Q^*$ as
\begin{equation}\label{sym_ext}
	\begin{aligned}
		\phi^*(y)=
		\begin{cases}
			\phi(y_1,2-y_2),\, 1<y_2<2,\\
			\phi(y_1,y_2),\,-1<y_2<1,\\
			\phi(y_1,-2-y_2),-2<y_2<-1.
		\end{cases}
	\end{aligned}
\end{equation}
Then the extended function $\phi^*$ satisfies
\be\label{ext_phi}
\begin{cases}
	\p_{y_1}^2\phi^*+a_2(y_1)\p_{y_2}^2\phi^*+a_1(y_1)\p_{y_1}\phi^*-a_0(y_1)(\ka+\phi^*(L_s,y_2))=\p_{y_1}f^*,\,\text{in}\,Q^*,\\
	\p_{y_1}\phi^*(L_s,y_2)-a_3(\ka+\phi^*(L_s,y_2))=g^*_1(y_2),\\
	\p_{y_1}\phi^*(L_1,y_2)=g^*_2(y_2),\\
	\p_{y_2}\phi^*(y_1,\pm 1)=\phi^*(L_s,-1)=0,\\
\end{cases}
\ee
\par Using the standard interior and the boundary estimates for the second order
linear elliptic equations in \cite{GT1983}, we obtain that $\phi^*(y)\in
C^{1,\al}(Q^*)$ and
\be\label{est1_phi_str}
\|\phi^*\|_{C^{1,\al}(Q^*)}\leq
C(\|f^*\|_{C^{1,\al}(Q^*)}+\|(g_1,g_2)\|_{C^{1,\al}[-2,2]}+\|\phi^*(L_s,y_2)\|_{L^2[-2,2]}+|\ka|),
\ee
which implies that $\phi^*(L_s,y_2)\in C^{1,\al}[-2,2]$. Use the estimate
(\ref{est1_phi_str}) again to conclude that $\phi^*(y)\in C^{2,\al}(Q^*)$ and
\be\label{est2_phi_str}
\|\phi^*\|_{C^{1,\al}(Q^*)}\leq
C(\|f^*\|_{C^{1,\al}(Q^*)}+\|(g_1,g_2)\|_{C^{1,\al}[-2,2]}+\|\phi^*(0,y_2)\|_{L^2[-2,2]}+|\ka|).
\ee
Then (\ref{c1_al_phi}) and (\ref{c2_al_phi}) follows immediately. \\
Step 2: Existence and Uniqueness of weak solutions. Due to the linearity, any
solution $\phi$ to the problem (\ref{gen_2_ellip}) can be decomposed as
$\phi=\phi_1+\phi_2$, where $\phi_i$, i=1,2, satisfy the following equation
respectively
\be\label{ihomo_phi1}
\begin{cases}
	\p_{y_1}^2\phi_1+a_2(y_1)\p_{y_2}^2\phi_1+a_1(y_1)\p_{y_1}\phi_1-a_0(y_1)\phi_1(L_s,y_2)=\p_{y_1}f,\,\text{in}\,Q,\\
	\p_{y_1}\phi_1(L_s,y_2)-a_3\phi_1(L_s,y_2)=g_1(y_2),\\
	\p_{y_1}\phi_1(L_1,y_2)=g_2(y_2),\\
	\p_{y_2}\phi_1(y_1,\pm 1)=0,
\end{cases}
\ee
and
\be\label{homo_phi2}
\begin{cases}
	\p_{y_1}^2\phi_2+a_2(y_1)\p_{y_2}^2\phi_2+a_1(y_1)\p_{y_1}\phi_2-a_0(y_1)(\ka+\phi_2(L_s,y_2))=0,\,\text{in}\,Q,\\
	\p_{y_1}\phi_2(L_s,y_2)-a_3(\ka+\phi_2(L_s,y_2))=0,\\
	\p_{y_1}\phi_2(L_1,y_2)=\p_{y_2}\phi_2(y_1,\pm 1)=0,\\
	\phi_2(L_s,-1)=-\phi_1(L_s,-1).
\end{cases}
\ee
Combing the Lemma \ref{sol_inhomo} with the Fredholm alternative theorem, one can
easily derive (\ref{ihomo_phi1}) has a unique $H^1(Q)$ solution $\phi_1$ which
satisfying (\ref{H1_est}). On the other hand, one can prove that the
only weak solution to (\ref{homo_phi2}) must be
$(\phi_2,\ka)=(-\phi_1(L_s,-1),\phi_1(L_s,-1))$ by applying the maximum principle.
For the detailed proof, one can refer to Lemma 4.1 in \cite{LXY2009b}. Thus the
proof is completed.
\end{proof}

At this point, we can easily illustrate the well-posedness to the reformulated
problem (\ref{phi_eq})-(\ref{phi_BC}).
\begin{lemma}\label{phi_rest}
The problem (\ref{phi_eq})-(\ref{phi_BC}) has a unique solution $(\phi,\ka)\in
C^{2,\al}(\bar Q)\times\mathbb{R}$ satisfying
\begin{equation}\label{phi_Ck_al}
	\begin{aligned}
		&\|\phi\|_{C^{k,\al}(\bar Q)}+|\ka|\leq C(\|(G_1,G_2)\|_{C^{k-1,\al}(\bar
		Q)}\\
		&+\|(R_6(y_2),R_7(y_2))\|_{_{C^{1,\al}[-1,1]}}+\|\epsilon\hat
		P_{ex}(y_2)\|_{C^{1,\al}[-1,1]}),\,k=1,2,
	\end{aligned}
\end{equation}
for some positive constant C.
\end{lemma}
\begin{proof}
It suffices to verify the solvability condition (\ref{sol_condi}) for the problem
(\ref{phi_eq})-(\ref{phi_BC}). A direct  but tedious computation shows that
\be\no
&&a_0(y_1)=-b_0\la_{0}(y_1)=-\frac{2 c^2(\bar\rho^+)\bar f
b_0b_3}{\bar u(c^2(\bar\rho^+)-\bar u^2)^2}>0,\\ \no
&&a_1(y_1)=\la_{1}(y_1)=\frac{\gamma \bar u^2+c^2(\bar\rho^+)}{(c^2(\bar\rho^+)-\bar u^2)^2}\bar f>0,\\ \no
&&a_2(y_1)=\frac{1}{1-\bar M^2}>0,\\ \no
&&a_3=b_0(b_2-\la(L_s))=b_0\frac{c^2(\bar\rho^+)(\bar\rho^+ -\bar\rho^-)\bar f(L_s)}{\bar\rho^+\bar u(c^2(\bar\rho^+)-\bar u^2)}>0,
\ee
since the background solution is subsonic and smooth, the upper bound is trivial.
Hence,
Proposition \ref{uniq_exis} implies that there exist a unique weak solution
$(\phi,\ka)$, and the estimate (\ref{phi_Ck_al}) follows
from  (\ref{c1_al_phi}) and (\ref{c2_al_phi}).
\end{proof}
In view of the analysis of the problem (\ref{phi_eq})-(\ref{phi_BC}), the well-posedness of the equation (\ref{v1v2-lin})-(\ref{v1v2_BC_lin}) follows.
\begin{lemma}\label{v1v2_lema}
The problem (\ref{v1v2-lin})-(\ref{v1v2_BC_lin}) admits a unique solution
$(v_1,v_2,v_4(-1))\in C^{2,\al}(\bar Q)\times \mathbb R$ satisfying
\begin{equation}\label{v1v2_est}
	\begin{aligned}
		&\|(v_1,v_2)\|_{C^{k,\al}(\bar Q)}+|v_4(-1)|\leq
		C(\|(G_1,G_2)\|_{C^{k-1,\al}(\bar Q)}\\
		&+\|(R_6(y_2),R_7(y_2))\|_{_{C^{k,\al}[-1,1]}}+\|\epsilon\hat
		P_{ex}(y_2)\|_{C^{k,\al}[-1,1]}),\,\,k=1,2,
	\end{aligned}
\end{equation}
and the compatibility conditions
\be\label{cap_condi}
\p_{y_2}v_1(y_1,\pm 1)=0,\,\,\p_{y_2}v_2(y_1,\pm 1)=\p^2_{y_2}v_2(y_1,\pm 1)=0.
\ee
\end{lemma}
\begin{proof}
Combine (\ref{rprstv1v2}) and (\ref{phi_Ck_al}), we conclude that there is a unique
solution $(v_1,v_2,v_4(-1))\in C^{1,\al}\times\mathbb{R}$ such that
\be\label{v1v2_low_est}
&&\|(v_1,v_2)\|_{C^{\al}(\bar Q)}+|\v_4(-1)|\leq C(\|(G_1,G_2)\|_{C^{\al}(\bar Q)}
\\\no
&&+\|(R_6(y_2),R_7(y_2))\|_{_{C^{1,\al}[-1,1]}}+\|\epsilon\hat
P_{ex}(y_2)\|_{C^{1,\al}[-1,1]}),
\ee
The similar estimates also holds true for $\|(v_1,v_2)\|_{C^{1,\al}(\bar Q)}$, but
we can derive an even better estimate by rewriting (\ref{v1v2-lin}) into an elliptic
equation of $v_1$. To this end, we firstly rewrite the system (\ref{v1v2-lin}) as
\be\label{v1v2_conser_form}
\begin{cases}
  \p_{y_1} v_1+\frac{1}{1-\bar M^2}\p_{y_2}v_2+\la_1(y_1)v_1=\mathcal{G}_1(y),\\
  \p_{y_1}v_2-\p_{y_2}v_1 =\mathcal{G}_2(y),\\
  v_1(L_s,y_2)=\mathcal{R}_6(y_2),\\
  v_1(L_1,y_2)=\mathcal{R}_7(y_2),\\
v_2(y_1,\pm 1)=0,
\end{cases}
\ee
where
\be\no
&&\mathcal{G}_1(y)=G_1(y)-\frac{B_3(y_1)b_3+B_4(y_1)}{(c^2(\bar\rho^+)
-(\bar u^+)^2)}(v_4(-1)+\int_{-1}^{y_2}b_0 v_2(L_s,\tau) d\tau),\\ \no
&&\mathcal{G}_2(y)=G_2(y)+b_0 v_2(L_s,y_2),\\ \no
&&\mathcal{R}_6(y_2)=b_2(v_4(-1)+\int_{-1}^{y_2}b_0v_2(L_s,\tau) d\tau)+R_6(y_2),\\ \no
&&\mathcal{R}_7(y_2)=\frac{b_3}{\bar u(L_1)}(v_4(-1)+\int_{-1}^{y_2}b_0v_2(L_s,\tau) d\tau)- \frac{\epsilon \hat P_{ex}(y_2)}{\bar u(L_1)}+R_7(y_2),\no
\ee
and
\be
\p_{y_2}\mathcal{G}_1(y_1,\pm 1)=0,\,\mathcal{G}_2(y_1,\pm 1)=0,\,\mathcal{R}'_6(\pm 1)=0,\,\mathcal{R}'_7(\pm 1)=0.
\ee
Note that (\ref{rprstv1v2}) and the boundary conditions of $\phi$ implies that
\be\no
v_2(y_1,\pm 1)=0,\,\p_{y_2}v_1(y_1,\pm 1)=0.
\ee

A simple cancellation yields to
\be\label{2order_v1}
\begin{cases}
\p_{y_1}((1-\bar M^2)\p_{y_1} v_1)+\p^2_{y_2}v_1+\p_{y_1}(\la_1(1-\bar M^2)v_1)=\p_{y_1}((1-\bar M^2)\mathcal{G}_1)-\p_{y_2}\mathcal{G}_2,\\
v_1(L_s,y_2)=\mathcal{R}_6(y_2),\,v_1(L_1,y_2)=\mathcal{R}_7(y_2),\\
\p_{y_2}v_1(y_1,\pm 1)=0,
\end{cases}
\ee
Since the boundary conditions in (\ref{2order_v1}) are compatible at the corners, we
obtain the following estimates for $v_1$ by using the symmetric extension defined by
(\ref{sym_ext})
\begin{equation}\label{est_v1_hig}
\begin{aligned}
\|v_1\|_{C^{k,\al}(\bar Q)}\leq&
\,C(\|v_2\|_{C^{k-1,\al}(\bar Q)}+|v_4(-1)|+\|(G_1,G_2)\|_{C^{k-1,\al}(\bar Q)}\\
&+\|\hat P_{ex}\|_{C^{k,\al}[-1,1]}+\|(R_6,R_7)\|_{C^{k,\al}[-1,1]}),\,k=1,2.
\end{aligned}
\end{equation}
This estimates together with (\ref{v1v2_conser_form}) also implies that
\begin{equation}\label{est_v2_hig}
\begin{aligned}
\|v_2\|_{C^{k,\al}(\bar Q)}\leq&
\,C(\|v_2\|_{C^{k-1,\al}(\bar Q)}+|v_4(-1)|+\|(G_1,G_2)\|_{C^{k-1,\al}(\bar Q)}\\
&+\|\hat P_{ex}\|_{C^{k,\al}[-1,1]}+\|(R_6,R_7)\|_{C^{k,\al}[-1,1]}),\,k=1,2.
\end{aligned}
\end{equation}
Hence, (\ref{v1v2_est}) is follows from (\ref{v1v2_low_est}) and (\ref{est_v1_hig})-(\ref{est_v2_hig}). Finally, differentiating the first equation in (\ref{v1v2_conser_form}) with respect to $x_2$ and  combining with  $\p_{y_2}\mathcal{G}_1(y_1,\pm 1)=0$ and $\p_{y_2}v_1(y_1,\pm 1)=0$, we obtain
\be\no
\p_{y_2}^2v_2(y_1,\pm 1)=0,
\ee
which gives the compatibility condition (\ref{cap_condi}). Thus, the proof is
complete.
\end{proof}
\section{A priori estimates and proofs of main results}

In this section, we will use the Banach contraction mapping theorem to prove Theorem
\ref{thm3.1}. Given any $\hat{V}\in \mathcal{V}_{\delta}$,  we could establish some
a priori estimates to the linearized problems defined in subsection \ref{iteration},
and construct a
contractible mapping from $\V_{\delta}$ into itself so that there exits a unique fixed point, which is the solutions obtained in Theorem \ref{thm3.1} and
the proof of Theorem \ref{thm3.1} will be finished.

Lemma \ref{v1v2_lema} implies that there is a unique solution $(v_1,v_2,v_4(-1))\in
C^{2,\al}(\bar Q)\times \mathbb R$ to the system (\ref{v1v2-lin})-(\ref{v1v2_BC_lin}) satisfying
\begin{equation}\label{est_v1v2_fin}
	\begin{aligned}
		&\|(v_1,v_2)\|_{C^{2,\al}(\bar Q)}+|v_4(-1)|\leq
		C(\|(G_1,G_2)\|_{C^{1,\al}(\bar Q)}\\
		&+\|(R_6(y_2),R_7(y_2))\|_{_{C^{2,\al}[-1,1]}}+\|\epsilon\hat
		P_{ex}(y_2)\|_{C^{2,\al}[-1,1]})\\
		&\leq C(\epsilon+\delta^2),
	\end{aligned}
\end{equation}
and the compatibility condition (\ref{cap_condi}) holds true.

The shock curve $v_4$ is given by (\ref{shoc_expr}), which satisfies
\be\label{est_v4_fin}
\|v_4\|_{C^{3,\al}[-1,1]}\leq C(|v_4(-1)|+\|v_2\|_{C^{2,\al}(\bar Q)}+\|F_5\|_{C^{2,\al}(\bar Q)})\leq C(\epsilon+\delta^2).
\ee
Moreover, it follows from (\ref{cap_condi}) and (\ref{err_F5}) that
\be\label{cap_codi_v4}
v_4'(\pm 1)=0=v_4^{(3)}(\pm 1).
\ee

It remains to solve $v_3$. Due to  (\ref{v3_expr}), combining with (\ref{err_F6}) and (\ref{cap_codi_v4}), we obtain the estimate
\be\label{est_v3_fin}
\|v_3\|_{C^{2,\al}(\bar Q)}\leq C(\|v_4\|_{C^{2,\al}[-1,1]}+\|F_6\|_{C^{2,\al}(\bar Q)})\leq C(\epsilon+\delta^2),
\ee
and the compatibility condition
\be\label{cap_codi_v3}
\p_{y_2}v_3(y_1,\pm 1)=0.
\ee
Taking $\delta=O(1)\epsilon$, then for any given $\hat V\in \V_{\delta}$ we can define a continuous mapping $T:\V_{\delta} \rightarrow \V_{\delta}$ as
\be\label{def_T}
T\hat V=V,
\ee
due to the iteration scheme introduced in the previous section and the estimates (\ref{est_v1v2_fin})-(\ref{cap_codi_v3}). Finally, we show that the mapping is also contractible in the space  $(C^{1,\al}(\bar Q))^3\times C^{2,\al}[-1,1]$.

For arbitrarily given two states $\hat V_i=(\hat v_1^i,\hat v_2^i,\hat v_3^i,\hat v_4^i)\in \V_{\delta},i=1,2$ with the corresponding fluid variable $(\hat u_1^i,\hat u_2^i,\hat B^i,\hat\xi^i)$, set
\be\no
V_i=T\hat V_i,\,i=1,2,
\ee
where $V_i=(v_1^i,v_2^i,v_3^i,v_4^i)$. For the convenience, we denote $\hat W=\hat V_1-\hat V_2$ and $W=V_1-V_2$, or equivalently,
\be\no
\hat w_k=\hat v_k^1-\hat v_k^2,\,\,w_k=v_k^1-v_k^2,\,\,1\leq k \leq 4.
\ee
The equation (\ref{shock3}) implies that
\be\label{diff_v4}
w_4'=b_0w_2+O(\epsilon)\sum_{i=1}^{4}\hat w_i,
\ee
which yields that
\be\label{est_diff_v4}
\|w_4'\|_{C^{1,\al}[-1,1]}\leq C\|w_2\|_{C^{1,\al}(\bar Q)}+C\epsilon (\sum_{i=1}^{3}\|\hat w_i\|_{C^{1,\al}(\bar Q)}+\|\hat w_4\|_{C^{1,\al}[-1,1]}).
\ee
It follows from (\ref{v3_expr}) that
\be\label{diff_v3}
w_3=b_3 w_4-b_3\int_{\beta_1(y)}^{\beta_2(y)}(\hat v_4^1(\tau))'d\tau+b_3\int_{y_2}^{\beta_2(y)}\hat w_4'(\tau)d\tau+O(\epsilon)\sum_{i=1}^{3}\hat w_i,
\ee
where $\beta_i,i=1,2$ is the initial position such that the corresponding characteristic $y_2^i(s,\beta_i)$ going through $(y_1,y_2)$ with $y_2^i(L_s)=\beta_i$. It is easy to verify that
\be\label{est_beta}
\|\beta_1(y)-\beta_2(y)\|_{C^{1,\al}(\bar Q)}\leq C(\|\hat w_1\|_{C^{1,\al}(\bar Q)}+\|\hat w_2\|_{C^{1,\al}(\bar Q)}+\|\hat w_4\|_{C^{2,\al}[-1,1]}),
\ee
thus,
\be\label{est_diff_v3}
\|w_3\|_{C^{1,\al}(\bar Q)}\leq C\| w_4\|_{C^{1,\al}[-1,1]}+C\epsilon(\sum_{i=1}^{3}\|\hat w_i\|_{C^{1,\al}(\bar Q)}+\|\hat w_4\|_{C^{2,\al}[-1,1]}).
\ee
It is straightforward to show that $w_1,w_2$ satisfies
\be\label{diff_v1v2}
\begin{cases}
  \p_{y_1} w_1+\frac{1}{1-\bar M^2}\p_{y_2}w_2+\la_1(y_1)w_1+\la_2(y_1)(w_4(-1)+\int_{-1}^{y_2}b_0w_2(L_s,\tau)d\tau)\\
 =\sum_{i=1}^{4}(O(\epsilon)\hat w_i+O(\epsilon)\nabla \hat w_i)+O(\epsilon)(\beta_1-\beta_2)+O(1)\int_{y_2}^{\beta_2(y)}\hat w_4'(\tau)d\tau,\\
  \p_{y_1}w_2-\p_{y_2}[w_1-\la(y_1) (w_4(-1)+\int_{-1}^{y_2}b_0w_2(L_s,\tau) d\tau)] \\
  =\sum_{i=1}^{4}(O(\epsilon)\hat w_i+O(\epsilon)\nabla \hat w_i)+O(\epsilon)\p_{y_2}(\beta_1-\beta_2)+O(1)\p_{y_2}\int_{y_2}^{\beta_2(y)}\hat w_4'(\tau)d\tau,
\end{cases}
\ee
and the boundary conditions
\be\label{BC_diff_v1v2}
\begin{cases}
  w_1(L_s,y_2)=b_2(\hat w_4(-1)+\int_{-1}^{y_2}b_0\hat w_2(L_s,\tau) d\tau)+\sum_{i=1}^{4}O(\epsilon)\hat w_i,\\
  w_1(L_1,y_2)=\frac{b_3}{\bar u(L_1)}(\hat w_4(-1)+\int_{-1}^{y_2}b_0\hat w_2(L_s,\tau) d\tau)+\sum_{i=1}^{4}
  O(\epsilon)\hat w_i\\
  +O(\epsilon)(\beta_1-\beta_2)(L_1,y_2)
  +O(1)\int_{y_2}^{\beta_2(L_1,y_2)}\hat w_4'(\tau)d\tau,\\
w_2(y_1,\pm 1)=0,
\end{cases}
\ee
where
\be\no
\la_2(y_1)=\frac{B_3(y_1)b_3+B_4(y_1)}{c^2(\bar\rho^+)-(\bar u^+)^2}.
\ee
Then, applying the estimate (\ref{v1v2_est}) to the system (\ref{diff_v1v2})-(\ref{BC_diff_v1v2}) with $k=1$ and together with (\ref{est_beta}), we obtain
\be\label{est_fiff_v1v2}
\|(w_1,w_2)\|_{C^{1,\al}(\bar Q)}+|w_4(-1)|\leq
C\epsilon(\sum_{i=1}^{3}\|\hat w_i\|_{C^{1,\al}(\bar Q)}+\|\hat w_4\|_{C^{2,\al}[-1,1]}).
\ee
Finally, collecting all these estimates above leads to
\be\label{est_diff_fin}
\sum_{i=1}^{3}\| w_i\|_{C^{1,\al}(\bar Q)}+\| w_4\|_{C^{2,\al}[-1,1]}\leq C\epsilon(\sum_{i=1}^{3}\|\hat w_i\|_{C^{1,\al}(\bar Q)}+\|\hat w_4\|_{C^{2,\al}[-1,1]}).
\ee
By \eqref{est_diff_fin}, there is a small constant $\epsilon_0$ such
that for all $\epsilon\in(0,\epsilon_0]$, the mapping $T$ defined by (\ref{def_T})
is  contractible in the Banach space  $(C^{1,\al}(\bar Q))^3\times C^{2,\al}[-1,1]$.
Therefore, there exists a unique solution $V$ in $(C^{1,\al}(\bar Q))^3\times
C^{2,\al}[-1,1]$. Due to the Lemma \ref{v1v2_lema} and the a priori estimates
(\ref{est_v4_fin}),(\ref{est_v3_fin}), we know that $V$ also belongs to
$\V_{\delta}$. It follows that $V$ satisfies the estimates (\ref{est_th3.1}). Thus ,
the proof of Theorem \ref{thm3.1} is complete. Theorem \ref{thm1} is a direct inference of Theorem \ref{thm3.1}. We omit the details.



\par {\bf Acknowledgement.} Weng is partially supported by National Natural Science Foundation of China 11701431, 11971307, 12071359.

\bibliographystyle{amsplain}

\end{document}